\documentclass[11pt]{amsart}

 \usepackage{mathrsfs}
\usepackage{amsmath}
\usepackage{amsfonts}
\usepackage{amssymb}
\usepackage{color}
\usepackage{amsthm}
\usepackage{graphicx}
\usepackage{fullpage}

\newtheorem{thm}{Theorem}[section]

\newtheorem{lem}[thm]{Lemma}
\newcounter{counter_conj-lem}
\newtheorem{conj-lem}[thm]{Conjecture-Lemma}

\newtheorem{prop}[thm]{Proposition}
\newtheorem{ex}[thm]{Example}

\newcounter{counter_conj-prop}
\newtheorem{conj-prop}[thm]{Conjecture-Proposition}

\newcounter{counter_conj-thm}
\newtheorem{conj-thm}[thm]{Conjecture-Theorem}

\newtheorem{cor}[thm]{Corollary}
\newtheorem{defi}[thm]{Definition}

\theoremstyle{remark}

\newtheorem{rem}[thm]{Remark}

\newtheorem{example}[thm]{Example}

\DeclareMathOperator{\LL}{L}

\DeclareMathOperator{\C}{\mathbb{C}}
\DeclareMathOperator{\N}{\mathbb{N}}

\DeclareMathOperator{\PGL}{PGL}
\DeclareMathOperator{\nefh}{Nef_{\mathbb{H}}}

\DeclareMathOperator{\R}{\mathbb{R}}

\DeclareMathOperator{\dtop}{d_{top}}

\DeclareMathOperator{\Pg}{\mathbb{P}}
\DeclareMathOperator{\ch}{cosh}

\DeclareMathOperator{\HH}{\mathbb{H}}

\DeclareMathOperator{\NS}{NS}
\DeclareMathOperator{\Bir}{Bir}
\DeclareMathOperator{\SL}{SL}

\title[A central limit theorem for Cremona transformations]{A central limit theorem for the degree of a random product of Cremona transformations}

\author{Nguyen-Bac Dang, Giulio Tiozzo}
\thanks{Stony Brook University, USA, \texttt{nguyen-bac.dang@stonybrook.edu}}
\thanks{University of Toronto, Canada, \texttt{tiozzo@math.utoronto.ca}}

\begin{document}

\begin{abstract}
We prove a central limit theorem for the algebraic and dynamical degrees of a random composition of Cremona transformations.
\end{abstract}

\date{February 22, 2021}

\maketitle

\section{Introduction}

A rational map $f : \Pg^2 \dashrightarrow \Pg^2$, defined over the field of complex numbers $\C$, is a function which is given in homogeneous coordinates by 
\begin{equation*}
f([x: y : z]) = [P_0(x,y,z):P_1(x,y,z):P_2(x,y,z)],
\end{equation*} 
where $P_0,P_1 $ and $P_2$ are homogeneous polynomials in $\C[x,y,z]$ of the same degree $d\in \mathbb{N}$ with no common factor. 
We say that $f$ is \textbf{dominant} if its image is not contained in an algebraic curve.
The integer $d$, denoted $\deg(f)$, is called the \textbf{degree} (or \textbf{algebraic degree}) of $f$ and is in general distinct from the \textbf{topological degree} of $f$, which is the number of preimages by $f$ of a general point. 
When the topological degree of $f$ is one, we say that $f$ is \textbf{birational} or that $f$ defines a \textbf{Cremona transformation}.
%A rational map $f$ of $\Pg^2$ is called a Cremona transformation, or a birational transformation if its topological degree is $1$.

Unlike the situation on $\Pg^1$, the algebraic degree of $f \circ g$ where $f$ and $g$ are two rational maps on $\Pg^2$ is not equal in general to the product $\deg(f) \deg(g)$, but it satisfies a submultiplicative property: 
\begin{equation} \label{E:submult}
\deg(f\circ g) \leqslant \deg(f) \deg(g).
\end{equation}
Using this fact, one can define the (\textbf{first}) \textbf{dynamical degree} of $f$ \cite{russakovskii_shiffman}, denoted $\lambda_1(f)$, given by 
\begin{equation} \label{E:lambda-one}
\lambda_1(f) := \lim_{n\rightarrow +\infty}\deg(f^n)^{1/n}.
\end{equation}
This dynamical quantity is invariant under birational 
conjugacy  \cite{dinh_sibony_une_borne_sup,tuyen2,dang2017degrees} and measures the growth rate of the preimages of a generic hyperplane on $\Pg^2$.   

\bigskip
The degree and the dynamical degree of an arbitrary composition are quite difficult to predict in general and this is due to the presence of points where the rational maps are not defined (called  \emph{indeterminacy points}) and their behavior under iteration (see \cite[Proposition 1.4.3]{sibony_1999}). 

As a result, the growth of the sequence $(\deg(f^n))$ and the dynamical degree of a given rational map $f$ has been the subject of much research, 
and is known only in certain cases: for endomorphisms of a projective variety, monomial maps  \cite{lin_algebraic_stability_and_degree_growth,
favrewulcandegree}, birational surfaces maps \cite{diller_favre,blanc_cantat}, polynomial automorphisms and endomorphisms of the affine plane \cite{friedland_milnor,furter,
valuative_tree,favre_jonsson_eigenvaluations,
favre_jonsson_dynamical_compactification}, meromorphic surface maps (under certain assumptions) 
\cite{boucksom_favre_jonsson_deggrowth},
birational transformations of hyperk\"ahler manifolds \cite{bianco}, certain automorphisms of the affine $3$-space \cite{blanc_santen_automorphism_degree_3,blanc_santen} and certain rational maps associated to matrix inversions \cite{maillard_growth_complexity,maillard_complexity,
 maillard_baxter_birational,bedford_kim_matrix_inversions,
 bedford_truong_matrix_inversion}.
  Starting from dimension $3$, the degree sequences are partially known for  birational transformations with very slow degree growth \cite{cantat_xie_degrees} or for specific examples \cite{deserti}, for a specific group of automorphisms on $\SL_2(\C)$ \cite{dang_tame}, while a lower bound on unbounded degree sequences was recently obtained for a large class of birational transformations in \cite{lonjou_urech}.

Note, however, that when $f,g$ are generic  maps (i.e. belong to suitable Zariski open subsets of the space of rational maps of degree $d$), the product 
satisfies $\deg(f\circ g)  = \deg(f) \deg(g)$. In other words, if $f, g$ are chosen ``randomly" enough, then $\lambda_1(f) = \deg(f)$ and the degree of a product behaves well. 
 This simple, but natural observation motivates the current paper. 

\bigskip
Let us fix  a probability measure $\mu$ on the space of Cremona transformations whose support is countable.  
    We consider the \textbf{random product} 
    $$f_n:= g_1 g_2 \ldots g_n$$ 
    where $(g_n)$ is a sequence of $i.i.d.$ random elements of $G$ chosen with distribution $\mu$; thus,  the sequence $(f_n)$ describes a \textbf{random walk} in the space of  birational maps of $\Pg^2$.
    
\bigskip
Our aim is to understand the distribution of the sequences of algebraic degrees $(\deg(f_n))$ and dynamical degrees $(\lambda_1(f_n))$. 
Heuristically, one expects the sequence  $(\log \deg(f_n))$ to be close to the sum $\log\deg (g_1) + \ldots + \log \deg (g_n)$. Since the logarithmic sum $\sum \log \deg(g_i)$ satisfies the classical law of large numbers and a central limit theorem,  the sequence $(\log \deg(f_n))$ should also display similar features.  
    
For the law of large numbers, observe that because of \eqref{E:submult} 
the sequence $(\log \deg (f_n))$ is subadditive and Kingman's theorem asserts that under a finite moment condition
 \begin{equation} \label{E:finite-log}
 \int_G \log \deg(f ) \ d\mu(f) < +\infty,
\end{equation}    
there exists a constant $\ell_\mu \geqslant 0$ such that %the sequence 
\begin{equation}
\lim_{n \to \infty} \dfrac{1}{n} \log \deg(f_n) = \ell_\mu
\end{equation}
almost surely. 
In other words, the sequence $\log \deg(f_n)$ follows a law of large numbers. 
This was shown in the case of rational maps on $\Pg^k$ in \cite{hindes2019dynamical}, where the quantity $\ell_\mu$ is referred to as the 
\emph{random dynamical degree}. 
For birational maps of $\Pg^2$, it was shown in \cite{maher_tiozzo_cremona} that the limit $\ell_\mu$ is positive if the semigroup generated by the support of $\mu$ is non-elementary.  Moreover, in that case the limit 
$$\lim_{n\rightarrow \infty}\frac{1}{n} \log\lambda_1(f_n)$$ 
exists almost surely (even though $(\log \lambda_1(f_n))$ is \emph{not} subadditive) and also equals $\ell_\mu$ whenever the support of $\mu$ is bounded. 

\bigskip
The various quantities described above have an explicit description when one restricts to monomial maps (\cite{lin_algebraic_stability_and_degree_growth,favrewulcandegree}).  
A monomial map  on $\Pg^k$ is a transformation of the form:
\begin{equation*}
g_M : [x_0: \ldots :x_k] \mapsto [x_0^{m_{00}} x_1^{m_{01}} \ldots x_k^{m_{0k}}: \ldots : x_0^{m_{k0}} x_1^{m_{k1}} \ldots x_k^{m_{kk}}],
\end{equation*}   
where $M:= (m_{ij})_{i\leq k, j\leq k}$ are the entries of a $(k+1) \times (k+1)$ matrix with integer coefficients. 
Using the fact that $g_{M N} = g_M \circ g_N$ for any two matrices $M,N$, a random product of monomial maps is induced by a random product of matrices. Since the degree of $g_M$ is a multiple of the largest coefficient of $M$ and the dynamical degree is the spectral radius of $M$, the convergence of $\frac{1}{n} \log \deg(f_n)$ follows from  
Oseledets' theorem \cite{oseledets}, while the central limit theorem for $\log \deg(f_n)$ can be deduced using Furstenberg-Kesten's results \cite{furstenberg_kesten_product_random_matrices}.   
Apart from the above special situation, very little is known for the behavior of the degree of a random product of birational transformations of $\Pg^k$ for $k\geq 3$. 

\bigskip

The main result of our paper shows that $\log \deg(f_n)$ for Cremona maps on $\Pg^2$ satisfies a central limit theorem, in the sense that the difference $\frac{\log \deg(f_n) -  n \ell_\mu}{\sqrt{n}}$ converges to either a Gaussian or a ``folded" Gaussian.
Let $\mu_n$ be the distribution of $f_n$, and let $C_b(\R)$ be the space of bounded, continuous functions $\varphi : \R \to \R$.
Given $\sigma \geqslant 0$, we denote as $\mathcal{N}_{\sigma}$ the \emph{Gaussian measure} of variance $\sigma$ and mean $0$, i.e. the probability measure 
$d\mathcal{N}_\sigma(t) = \frac{1}{\sqrt{2\pi} \sigma} e^{-\frac{t^2}{2 \sigma^2}} dt$ if $\sigma > 0$, and the $\delta$-mass at $0$ if $\sigma = 0$. 
Given $\sigma >0$, we define the \emph{folded Gaussian measure} centered at $0$ and of variance $\sigma$ as the probability 
measure $\mathcal{FN}_{\sigma}$ 
on $\mathbb{R}$ defined as the pushforward of a Gaussian measure of mean $0$ and variance $\sigma$ under the map $x \mapsto |x| $. 

\medskip

\noindent \textbf{Theorem A.}
 \label{T:main}
Let $G$ be a countable semigroup of Cremona transformations and let $\mu$ be a measure whose support generates $G$ satisfying
\begin{equation} \label{E:moment2}
\int_G {\sqrt{\deg(f)}} \ d \mu(f) < + \infty.
\end{equation}
Then there exists $\ell \geqslant 0$ such that  
the following two properties hold.
\begin{enumerate}
\item[(i)] (CLT for algebraic degree)
Either, there exists $\sigma \geqslant 0$ such that 
\begin{equation*}
\lim_{n\rightarrow +\infty} \int_G \varphi \left ( \dfrac{\log \deg(f) - n \ell}{\sqrt{n}} \right ) d\mu_n(f) = \int_{\R} \varphi(t) \  d \mathcal{N}_\sigma(t) 
\end{equation*}
for any function $\varphi \in C_b(\mathbb{R})$, or 
\begin{equation*}
\lim_{n\rightarrow +\infty} \int_G \varphi \left ( \dfrac{\log \deg(f) - n \ell}{\sqrt{n}} \right ) d\mu_n(f) =  \int_{\R} \varphi(t) \ d \mathcal{FN}_{\sigma}(t) 
\end{equation*}
for any $\varphi \in C_b(\mathbb{R})$.

\item[(ii)] (CLT for dynamical degree)
A similar limit law holds for the dynamical degree. 

There exists $\sigma \geqslant 0$ such that either  
\begin{equation*}
\lim_{n\rightarrow +\infty} \int_G \varphi \left ( \dfrac{\log \lambda_1(f) - n \ell}{\sqrt{n}} \right ) d\mu_n(f) =  \int_{\R} \varphi(t) \  d \mathcal{N}_\sigma(t) 
\end{equation*}
for any $\varphi \in C_b(\mathbb{R})$, or 
\begin{equation*}
\lim_{n\rightarrow +\infty} \int_G \varphi \left ( \dfrac{\log \lambda_1(f) - n \ell}{\sqrt{n}} \right ) d\mu_n(f) =  \int_{\R} \varphi(t) \ d \mathcal{FN}_{\sigma}(t) 
\end{equation*}
for any $\varphi \in C_b(\mathbb{R})$.

\noindent Finally:
\item[(iii)]
If the semigroup $G$ is non-elementary, then $\sigma > 0$ unless $G$ has arithmetic length spectrum. 
\end{enumerate}

\medskip

Let us remark that condition \eqref{E:moment2} is only needed if the semigroup generated by the support of $\mu$ is  parabolic, while the weaker condition \eqref{E:finite-log} is sufficient otherwise. A more refined version of this result, with a classification of all cases, formulas for $\ell$ and $\sigma$, as well as a characterization of the cases where $\sigma = 0$ will be given in \S~\ref{section_summary_elementary}.  

%Finally, in the non-elementary case, the fact that $\sigma = 0$ is related to the arithmeticity of the length 
%spectrum of $G$. 

When one considers groups of isometries of finitely dimensional hyperbolic spaces, there are no non-elementary subgroups which have arithmetic length spectrum 
(\cite{Dalbo}, \cite{kim-arith}). However, we show that such examples \emph{do} exist in the Cremona group: 
%there exist non-elementary semigroups of the Cremona group which have arithmetic length spectrum, and so that $\sigma = 0$ in the central limit theorem.

\begin{prop}  \label{P:exist-arith}
There exist non-elementary subgroups in the Cremona group which have arithmetic length spectrum. Consequently, there exist random walks supported 
on non-elementary semigroups for which $\sigma = 0$ in the central limit theorem.
\end{prop}

%when the support of the measure $\mu$ is elementary. 

\medskip
Our result is reminiscent of similar statements, proved in other contexts.
The central limit theorem is known for the norm of a random product of $n\times n$ matrices \cite{furstenberg_kesten_product_random_matrices,
guivarch_lepage_raugi,
guivarch_raugi_frontiere,
guivarch_raugi_product,
benoist_quint_central}, for the translation length and the escape rate for a random composition of isometries on a tree \cite{nagnibeda_woess} or more generally on a Gromov-hyperbolic space \cite{benoist_quint_hyperbolic,mathieu_sisto}, for quasimorphisms on a random product of elements in a countable hyperbolic group \cite{calegari-fujiwara, bjorklund-hartnick}, for the distance in the Teichm\"uller metric of a random product of mapping classes  \cite{horbez}.

\medskip
Let us now observe that the second case of Theorem A (with the folded normal law) actually occurs. 
Take a H\'enon map $h: \C^2 \to \C^2$ of degree $d$, of the form:
\begin{equation*}
h(x,y) = (y + P(x), x ),
\end{equation*}  
where $P(x) \in \C[x]$ is a polynomial of degree $d$ and let us take $\mu =  \frac{1}{2}\delta_h  +  \frac{1}{2}\delta_{h^{-1}}$, putting uniform mass on $h$ and $h^{-1}$. 
Since $\log \deg(h^{p}) = \log \lambda_1(h^p) = |p| \log(d)$ for all $p \in \mathbb{Z}$, the classical central limit theorem yields the convergence:
\begin{equation*}
\lim_{n\rightarrow +\infty }\dfrac{1}{\sqrt{n}} \int_G \phi(\log \deg(f) ) d\mu_n(f) = \lim_{n\rightarrow +\infty }\dfrac{1}{\sqrt{n}} \int_G \phi(\log \lambda_1(f) ) d\mu_n(f) =\int_{\R} \phi(|t|) e^{-t^2 / 2} dt, 
\end{equation*} 
for any bounded continuous function $\phi \in L^1(\R)$.
 In this situation, the logarithm of the degree of a random product does not converge to a normal law but to a folded normal law  and satisfies the second assertion of Theorem A. This example is an analogue in this setting of Furstenberg-Kesten's \cite[Example 2]{furstenberg_kesten_product_random_matrices} for products 
 of random matrices, where the folded normal law already appears.
Further concrete examples of the different asymptotic behaviours are given in Section \ref{S:examples}.
  
\medskip
Our proof exploits in a crucial way the relationship between birational maps and a suitable isometric action on an infinite dimensional Gromov-hyperbolic space or Hilbert space, developed by Cantat, Boucksom-Favre-Jonsson, 
Blanc-Cantat \cite{cantat_bir_surfaces,
 boucksom_favre_jonsson_deggrowth,blanc_cantat}. 
 The construction of this Gromov-hyperbolic space, denoted $\HH^\infty$, is of algebraic nature: it is obtained by 
 considering a subspace of divisors on the space of infinite blow-ups of $\Pg^2$ and taking its completion with respect to a norm induced by the intersection product. The Hodge index theorem guarantees that the intersection product  on $\HH^\infty$ defines a Lorentzian metric:
 \begin{equation*}
 d(\alpha, \beta) := \cosh^{-1}(\alpha \cdot \beta),
 \end{equation*}
 where $\alpha,\beta \in \HH^\infty$ and $(\alpha \cdot \beta)$ denotes the intersection product of $\alpha$ and $\beta$.
 One advantage in working on the space of divisors over all blow-ups of $\Pg^2$ is that the pullback action by a birational map $f$ becomes functorial. Namely, if $\alpha \in \HH^\infty$, and $f,g$ are birational maps, then:
 \begin{equation*}
 (f\circ g)^* \alpha = g^* f^* \alpha,
 \end{equation*}
 as if we were working with the action of an endomorphism on the N\'eron-Severi group of a surface. 
  Cantat exploited the fact that the pullback action on $\HH^\infty$ is an isometry to show  that the Tits alternative holds for the Cremona group \cite{cantat_bir_surfaces}. 
Thus, we obtain a representation $\rho_f :=  f^* $ from the Cremona group  to the group of isometries of $\HH^\infty$. 
Taking as $L$ the class of a line in $\Pg^2$, a random walk $(f_n)$ on the space of birational maps induces a sample path $\rho_{f_n}(L)$ in the hyperbolic space $\HH^\infty$ and we relate the degree to the distance on this space:
\begin{equation*}
\cosh d(\rho_{f_n}(L), L) = {\deg(f_n)}.
\end{equation*}

Denote by $G$ the semigroup of birational transformations generated by the support of $\mu$. 
According to the classification \cite{gromov_hyperbolic} 
of semigroups of isometries of a hyperbolic space, $\rho(G)$ is either non-elementary or elementary, which is further subdivided into elliptic, parabolic, focal, or lineal  (see Section \ref{S:examples}). 

In the non-elementary case, $\rho(G)$ contains two loxodromic elements with different axis; here, we import into our setting the known central limit theorems for the translation length and for the escape rate from \cite{bjorklund,mathieu_sisto, benoist_quint_hyperbolic,gouezel_analycity,horbez}. 

Otherwise, $\rho(G)$ either contains no loxodromic elements (elliptic or parabolic case) or every loxodromic element has a common fixed point on the (Gromov) boundary of $\HH^\infty$ (focal or lineal case). 
Here, we show that $G$ is particularly rigid and we conclude using some techniques from birational geometry by showing that the logarithm of a random product can be reduced to a random walk on the line or on the plane $\mathbb{R}^2$.

The possible limit distributions in the various cases are summarized in the following table. Note that the presence of a folded Gaussian 
implies that $G$ is lineal.

\bigskip

\begin{tabular}{|l|l|l|}
\hline
\textbf{Type of isometry group} & \textbf{Limit law for $\log \deg$} & \textbf{Limit law for $\log \lambda_1$} \\
\hline
elliptic & Gaussian (possibly trivial) & Gaussian (possibly trivial)  \\
\hline
parabolic & Gaussian (possibly trivial) & Gaussian (possibly trivial) \\
\hline
%focal &  Gaussian or Folded Gaussian & Gaussian or Folded Gaussian \\
%\hline
lineal & Gaussian or Folded Gaussian  & Gaussian or Folded Gaussian \\
\hline
non-elementary & Gaussian ($\ell > 0$) & Gaussian  ($\ell > 0$) \\
\hline
\end{tabular}

\section*{Acknowledgements }

We thank Jeffrey Diller, Romain Dujardin, Charles Favre, Junyi Xie for useful conversations, the first author's sister Nguyen-Thi Dang for providing us references and Mattias Jonsson for making us meet during our visit
to the University of Michigan. 
G. T. is partially supported by NSERC and the Alfred P. Sloan Foundation. 

\section{Rational maps, degrees and isometric actions}

\subsection{Topological, algebraic and dynamical  degrees}
  
Let $f : \Pg^2 \dashrightarrow \Pg^2$ be a rational map. 
The map $f$ can be expressed in homogeneous coordinates as:
\begin{equation*}
f([x: y : z]) = [P_0(x,y,z):P_1(x,y,z):P_2(x,y,z)],
\end{equation*}
where $P_0 , P_1,P_2 \in \C[x,y,z]$ are homogeneous polynomials of the same degree $d \in \N$ with no common factor.
We call $f$ \textbf{dominant} if the image of $\Pg^2$ is not contained in an algebraic curve. 

The integer $d$ is called the \textbf{degree} or the \textbf{algebraic degree} of the rational map $f$ and is denoted $\deg(f)$.  
 One can show (\cite{russakovskii_shiffman}) for any dominant rational map $f, g : \Pg^2 \dashrightarrow \Pg^2$ that:
 \begin{equation*}
 \deg(f\circ g) \leqslant \deg(f) \deg(g),
 \end{equation*}
so the sequence $(\deg(f^n))_{n \geqslant 1}$ for a given rational map $f$ is submultiplicative.  
We thus define the \emph{(first) dynamical degree} of $f$ as 
$$\lambda_1(f) := \lim_{n \to \infty} \deg(f^n)^{1/n}.$$

Recall that given a dominant rational map $f : \Pg^2 \dashrightarrow \Pg^2$, the \textbf{topological degree} of $f$, denoted $\dtop(f)$, is the number of preimages, counted with multiplicity, of a generic point of $\Pg^2$.
When the topological degree of $f$ is equal to $1$, one says that $f$ is \textbf{birational} and its inverse is a rational map which we denote by $f^{-1}$. In this paper, we will restrict to birational transformations of $\Pg^2$, which are often referred to as \textbf{Cremona transformations} of the plane.
Denote by $L$ the divisor on $\Pg^2$ given by the line at infinity. One can express the topological degree and the degree of $f$ by computing the following intersection products.
\begin{equation*}
\dtop(f) =1 =  (f^* L \cdot f^*L),
\end{equation*}
\begin{equation*}
\deg(f) = (f^* L \cdot L).
\end{equation*}

The dynamical degree and the topological degree are dynamical invariants and their properties are stated in the following result.

\begin{thm}   \label{T:ineq} 
\textup{(\cite{russakovskii_shiffman}, \cite[Theorem 1.1]{tuyen2}, \cite[Theorem 1]{dang2017degrees})} For any birational map $g : \Pg^2 \dashrightarrow \Pg^2$,  
$$\lambda_1(f) = \lambda_1(g \circ f \circ g^{-1}).$$ 
\end{thm}

\subsection{The construction of the hyperbolic space}

In this section, we recall the construction of the Picard-Manin space of divisors, following closely the presentation in \cite{blanc_cantat}.
We start with $X_0 = \mathbb{P}^2$. If $\pi: X \to X_0$ is a birational morphism, we say that $X$ is a \textbf{birational model of $X_0$}. When this happens, the morphism $\pi$ induces a pullback in the N\'eron-Severi group
$$ \pi^* : \NS(X_0) \to \NS(X). $$

Moreover, for any two birational models $ X,Y $ over $X_0$, there exists a third birational model $Z$ over both $X$ and $Y$. 
We thus define the Picard-Manin space as the inductive limit:
\begin{equation*}
\mathcal{Z} := \varinjlim{} \NS(X),
\end{equation*}
where $X$ describes all birational models of $X_0$. 
If $X$ is a blow-up of $X_0$ at one point, we denote by $E$ the exceptional divisor on $X$ and $\NS(X) \simeq \NS(X_0) \oplus \mathbb{Z} E$. 
If one takes an arbitrary sequence of blow-ups of $\Pg^2$, we obtain finitely many exceptional divisors which are all inside $\mathcal{Z}$. 
The Picard-Manin space can be described as:
\begin{equation*}
\mathcal{Z} = \NS(\Pg^2) \oplus \bigoplus \mathbb{Z} E_i \simeq \mathbb{Z} L  \oplus \bigoplus \mathbb{Z} E_i,
\end{equation*}
 where $E_i$ describes all the exceptional divisors on a birational model of $\Pg^2$ and where $L$ denotes the class of a line in $\Pg^2$.
 
 The intersection product on each birational model of $\Pg^2$ induces a scalar product on $\mathcal{Z}$, denoted $(\alpha \cdot \beta )$, and a norm on $\mathcal{Z} \otimes \R$. We denote by $\overline{ \mathcal{Z}}$ the completion of $\mathcal{Z}$ with respect to this norm.
Observe that the Hodge index theorem on each birational model of $\Pg^2$ shows that the metric induced by the intersection product is hyperbolic; as a result, the space $\overline{\mathcal{Z}}$ endowed with the metric induced by the intersection product has the structure of an infinite-dimensional hyperbolic space. For more details on this construction, we shall refer to \cite{cantat_bir_surfaces}. 

\begin{defi} The hyperbolic space $\HH^\infty$ is the set
\begin{equation*}
\HH^\infty: = \left \lbrace \alpha \in \overline{\mathcal{Z}} \ : \ (\alpha \cdot \alpha ) = 1 , (\alpha \cdot L) > 0 \right \rbrace. 
\end{equation*}
It is endowed with a hyperbolic metric $d: \HH^\infty \times \HH^\infty \to \R^+$  given by the formula:
\begin{equation*}
d(\alpha, \beta) := \cosh^{-1} (\alpha \cdot \beta),
\end{equation*}
for any $\alpha, \beta \in \HH^\infty$.
Its boundary, denoted $\partial \HH^\infty$, is the set
\begin{equation*}
\partial \HH^\infty := \left \lbrace  \alpha \in \overline{\mathcal{Z}} \ : \ (\alpha\cdot \alpha)=0, (\alpha \cdot L)>0 \right \rbrace.
\end{equation*}
\end{defi}

This space corresponds to the choice of a ``positive" hyperboloid in the vector space $\overline{\mathcal{Z}}$. 
We will restrict our action to a smaller subset of $\HH^\infty \cup \partial \HH^\infty$, namely the nef cone. 
Recall that a class is \emph{nef} if it intersects non-negatively any curve class and that a nef class is also \emph{big} if its self-intersection is positive.

\begin{defi} The nef locus $\nefh$ is the subset of nef classes in $\HH^\infty$ and we denote by $\partial \nefh$  the set of nef classes on the boundary  $ \partial \HH^\infty$.
\end{defi}

\subsection{Isometric action of birational maps on the hyperbolic Picard-Manin space}

Let $f : \Pg^2 \dashrightarrow \Pg^2$ be a birational map. 
Its graph $\Gamma_f$ in $\Pg^2 \times \Pg^2$ is a natural birational model of $\Pg^2$ and the maps $\pi_1,\pi_2$ induced by the projection onto the first and second factor, respectively, are regular. 
If $\alpha $ is a divisor in $\Pg^2$, then we can take its pullback by $\pi_2$, denoted $f^* \alpha$. More generally, we can do the same if $\alpha$ is a class in a birational model $X$ of $\Pg^2$ by pulling back on the corresponding graph. 
The latter definition is compatible with the inductive definition and induces a continuous pullback map $f^* : \HH^\infty \to \HH^\infty$. 

We now define a contravariant action by Cremona transformations on $\HH^\infty$, namely for any $\alpha \in \HH^\infty$ and any  birational map $f$, the element 
$\rho_f(\alpha)$ is given by the formula
\begin{equation*}
\rho_f(\alpha) :=  f^*\alpha.
\end{equation*}

Note that since $(f\circ g)^* = g^* \circ f^*$, this action  reverses the order.
Using the fact that $(f^*\alpha \cdot f^* \beta) =  (\alpha \cdot \beta)$ for all $\alpha ,\beta\in \HH^\infty$, one verifies that the above action induces an isometry of $(\HH^\infty,d)$ and the inverse is induced by $(f^{-1})^*$. 

As a consequence, if the associated isometry is loxodromic on $\HH^\infty$ then we relate the dynamical degree of $f$ with the translation distance as follows.

\begin{lem} \label{lem_distance_log}
For any birational map $f$ on $\Pg^2$, we have 
$$\log\left (f^*L \cdot L \right )  \leqslant d(\rho_f(L), L) \leqslant\log\left ( 2 {(f^*L \cdot L)}  \right).$$ 
\end{lem}

\begin{proof}
By definition of the hyperbolic metric, 
$$\cosh d(\rho_f(L), L) = {(f^*L \cdot L)} $$
hence, since $\frac{1}{2} e^x \leqslant \cosh x \leqslant e^x$, 
$$ \frac{e^{d(\rho_f(L), L)}}{2} \leqslant  {(f^*L \cdot L)} \leqslant e^{d(\rho_f(L), L)}$$
which immediately yields the claim.
\end{proof}

Observe that when the action of $f$ is loxodromic, then the two invariant classes on the boundary $\partial \HH^\infty$ are also in $\partial\nefh$.
We shall use frequently the following observation. 

\begin{lem} \label{trick_lemma} Let $G$ be a semigroup of  birational maps of $\Pg^2$. Suppose that there exists an element $\alpha \in \HH^\infty \cup \partial \HH^\infty$ which is an eigenvector for every element of $G$. Then the map $\pi : G \to (\R, +)$
\begin{equation*}
 \pi(f) := \log (f^* \alpha \cdot L)  - \log(\alpha\cdot L)
\end{equation*}
is a morphism of semigroups. 
\end{lem}

In the following sections, we will use many  times the following result. 
\begin{lem} \label{lem_big} If $\alpha \in \HH^\infty$ is  big and nef then there exists a constant $C$ such that for all $f\in G$, one has
\begin{equation} \label{eq_deg_fix_point}
|\log (f^* L \cdot L) - \log(f^* \alpha \cdot L)| \leqslant C.
\end{equation}
\end{lem}

\begin{proof}
Since $\alpha$ and $L$ are big and nef, Siu's inequality (see \cite{trapani},\cite[Theorem 2.2.13]{lazarsfeld_positivity_1}) yields:
\begin{equation*}
\dfrac{(L^2)}{2(\alpha \cdot L)} f^* \alpha \leqslant  f^* L \leqslant 2\dfrac{(\alpha \cdot L)}{(\alpha^2)} f^* \alpha,
\end{equation*}
where $\alpha \leqslant \beta $ means that the difference $\beta - \alpha$ lies in the closure of the cone generated by effective curves.
Intersecting with the nef class $L$ thus yields:
\begin{equation*}
\dfrac{(L^2)(f^*\alpha \cdot L)}{2(\alpha \cdot L)} \leqslant \deg(f) = (f^* L \cdot L) \leqslant 2 \dfrac{(\alpha \cdot L)(f^*\alpha \cdot L)}{(\alpha^2)}.
\end{equation*}
which implies the claim.
\end{proof}

\subsection{Classification of semigroups of isometries} \label{S:examples}

By the classification of semigroups of isometries of hyperbolic spaces (see \cite[Theorem 6.2.3 and Proposition 6.2.14]{das_simmons_urbanski}), a semigroup $G$ acting by isometry on a hyperbolic space $X$ satisfies one of the following properties:

\begin{enumerate}
\item[(i)] $G$ is \emph{elliptic}, i.e. there exists a class $\alpha \in X$ globally fixed by $G$.
\item[(ii)] $G$ is \emph{parabolic}, i.e. there exists a class $\alpha \in \partial X$ globally fixed by $G$ and every element of $G$ is parabolic.
\item[(iii)] $G$ is \emph{focal}, i.e. it globally fixes a class $\alpha \in \partial X$ and contains a hyperbolic element.
\item[(iv)] $G$ is \emph{non-elementary}, i.e. there exists two hyperbolic elements whose fixed sets at infinity do not intersect.
\item[(v)] $G$ is \emph{lineal}, i.e. it contains a hyperbolic element and any other hyperbolic element fixes the same points at infinity.
\end{enumerate} 

We call a semigroup $G$ \emph{elementary} if it satisfies condition $(i), (ii), (iii)$ or $(v)$ in the above characterization. 
In the situation where $G$ is a semigroup of Cremona transformations, the above classification yields.

\begin{prop} Let $G$  be a semigroup of Cremona transformations.  Then one of the following properties hold.
\begin{enumerate}
\item[(i)] $G$ is \emph{elliptic}, i.e. there exists a class $\alpha \in \HH^\infty$ globally fixed by $G$.
\item[(ii)] $G$ is \emph{parabolic}, i.e. there exists a class $\alpha \in \partial\nefh$ globally fixed by $G$ and every element of $G$ is parabolic.
%\item[(iii)] $G$ is \emph{focal}, i.e. it globally fixes a class $\alpha \in \partial \nefh$ and contains a hyperbolic element.
\item[(iii)] $G$ is \emph{non-elementary}, i.e. there exists two hyperbolic elements whose fixed sets at infinity do not intersect.
\item[(iv)] $G$ is \emph{lineal}, i.e. it contains a hyperbolic element and any other hyperbolic element fixes the same points at infinity.
\end{enumerate} 
\end{prop}

\begin{proof}
By \cite[Lemma 7.3]{urech2018subgroups}, $G$ cannot induce a focal subgroup of isometries on $\HH^\infty$. So we are left with the four remaining cases of the classification of isometries.
\end{proof}

\bigskip
We now give some concrete examples of subgroups in each of these classes and discuss the central limit theorem. 

\begin{example} If $G \subset \PGL_2(\C)$ is a discrete subgroup acting linearly on $\Pg^2$, then the semigroup induced by $G$ on $\HH^\infty$ is elliptic. In this case, the degrees and dynamical degrees are always $1$ and the sequence $\log \deg(f_n)$ is the constant random variable equal to zero.
\end{example}

%We then give an example of elliptic semigroup containing non-invertible elements. 
\begin{example} If $G$ is a family of non-trivial Jonqui\`eres transformations, i.e. of the form:
\begin{equation*}
(x,y) \mapsto  \left (ax + b , \dfrac{\alpha(x) y + \beta(x)}{\gamma(x) y + \delta(x)} \right ),
\end{equation*}
where $a \in \C^* , b\in \C$ and $\alpha,\beta, \gamma, \delta \in \C(x)$ are non-constant rational functions on $x$ and $\alpha \delta - \beta \gamma$ is a non-zero function. 
Then the subgroup $G$ induces a parabolic action on $\HH^\infty$. 
\end{example}

\begin{example} Take $h$ a H\'enon map, i.e. of the form:
\begin{equation*}
h: (x,y) \mapsto (y + P(x), x ),
\end{equation*}
where $P(x) \in \C[x]$ is a polynomial of degree $d\geqslant 2$. 
Consider the measure $\mu = \frac{1}{2} \delta_h + \frac{1}{2} \delta_{h^{-1}}$. 
The subgroup generated by $h$ and its inverse induces a lineal subgroup of isometries on $\HH^\infty$.
Since $\deg(h^p) = d^{|p|}$ for all $p\in\mathbb{Z}$, we have that $\frac{1}{\sqrt{n}} \log \deg (f_n)$ follows a folded normal law.
\end{example}

\subsection{Characterization of semigroups having a global fixed point on the boundary}
In this section, we study the semigroup of birational transformations whose action on the Picard-Manin space has a global fixed point on the boundary.
In many cases, we shall use the following result. 

Recall that the \emph{translation length} of an isometry $f$ of a hyperbolic metric space $(X, d)$ is 
$$\tau(f) := \lim_{n \to \infty} \frac{d(o, f^n o)}{n},$$
where $o \in X$ is any base point.  Moreover the isometry $f$ is \emph{loxodromic} if $\tau(f) > 0$. 

\begin{prop} \label{prop_kernel} Take a semigroup $G$ and suppose that there exists a class $\alpha \in \partial \nefh$ which is fixed by $G$ and such that $f^* \alpha = \lambda(f) \alpha$ for $\lambda(f)\in \R^*$. Then the following properties are equivalent.
\begin{enumerate}
\item[(i)]  One has $\lambda(f) = 1$, 
\item[(ii)] The action of $f$ on $\HH^\infty$ is not loxodromic.
\item[(iii)] One has $\lambda_1(f) = 1$.
\end{enumerate}
\end{prop}

This proposition yields the following corollary:

\begin{cor} \label{cor_kernel} Take a semigroup $G$ and suppose that there exists a class $\alpha \in \partial \nefh$ which is fixed by $G$ and that $f^* \alpha = \lambda(f) \alpha$ for $\lambda(f)\in \R^*$. 
Then 
\begin{equation} \label{E:lambda1}
\lambda_1(f) = \max(\lambda(f), \lambda(f)^{-1})
\end{equation}
for all $f\in G$.
\end{cor}

Before proving the above statement we will need the following lemma.

\begin{lem} \label{lem_lambda_deg} 
For any $f$ in $G$, one has
\begin{equation*}
\max \left ( {\lambda(f)}, {\lambda(f)}^{-1} \right ) \leqslant 2  {\deg(f)}.
\end{equation*}
\end{lem}

\begin{proof}
Let us prove the inequality ${\lambda(f)} \leqslant 2 {\deg(f)}$.
Since $\alpha $ is nef, we have by Siu's inequality 
\begin{equation*}
(f^* \alpha \cdot L) \leqslant 2  \dfrac{(\alpha \cdot L)}{(L^2)} (f^* L \cdot L).
\end{equation*}
Hence, since $f^*\alpha= \lambda(f) \alpha$, we obtain $\lambda(f) \leqslant 2 \deg(f)$ as required.
\medskip

For the second inequality, we compute the intersection product $(f_* \alpha \cdot L)$ and obtain:
\begin{equation*}
(f_* \alpha \cdot L) = \dfrac{1}{\lambda(f)} (f_* f^* \alpha \cdot L ) = \dfrac{1}{\lambda(f)} (\alpha \cdot L).
\end{equation*}
Moreover, the projection formula shows that:
\begin{equation*}
(f_* \alpha \cdot L) = (\alpha \cdot f^* L).
\end{equation*}
By Siu's inequality, we have:
\begin{equation*}
f^* L \leqslant 2  \dfrac{(f^* L \cdot L)}{(L^2)} L,
\end{equation*}
and using the fact that $\alpha$ is nef, we obtain:
\begin{equation*}
\dfrac{1}{\lambda(f)} (\alpha \cdot L) = (\alpha \cdot f^* L) \leqslant 2 \deg(f) (\alpha \cdot L).
\end{equation*}
Dividing by $(\alpha \cdot L)$  yields the second inequality.
\end{proof}

One important result is the following lemma, which provides good estimates on the degree.

\begin{lem} \label{lem_cantat_blanc} Suppose that $\alpha \in \partial \HH^\infty \cap \partial \nefh$ and that $f^*\alpha = \alpha$ for any $f\in G$.  Then for any $f,g \in G$, one has:
\begin{equation*}
\sqrt{{\deg(f\circ g)}} \leqslant  \sqrt{{\deg(f)}} +  \sqrt{{\deg(g)}} . 
\end{equation*}
\end{lem}

 \begin{proof}[Proof of Lemma \ref{lem_cantat_blanc}]
By rescaling $\alpha$, let us assume $(\alpha \cdot L ) = 1$. Given $f, g \in G$, we write:
\begin{equation*}
f^* L = \deg(f) \alpha + v_1,
\end{equation*} 
and 
\begin{equation*}
g_* L = \deg(g) \alpha + v_2
\end{equation*}
where $v_1, v_2 \in \HH^\infty$ and $(v_i \cdot L) =0$.
Using the projection formula, the fact that $(\alpha^2) = 0$ and our decomposition, we have:
\begin{equation} \label{eq_normalization_1}
1 = (\alpha \cdot L) = (g^*\alpha \cdot L) = (\alpha \cdot g_* L)=  (\alpha \cdot v_2).
\end{equation}
Similarly since $ f_* \alpha =  \alpha$, we also have: 
\begin{equation} \label{eq_normalization_2}
1 = (\alpha \cdot L) =  (f_*\alpha \cdot L) = (\alpha \cdot f^* L) =  (\alpha \cdot v_1).
\end{equation}
Let us also compute $(f^*L \cdot f^*L)$ and $(g_* L \cdot g_* L)$, 
\begin{equation} \label{eq_self_intersection_1}
1 = (f^*L \cdot f^*L) = 2 \deg(f)(\alpha \cdot v_1) + (v_1^2),
\end{equation}
and
\begin{equation*} 
 (g_*  L \cdot g_* L) = 2 \deg(g)(\alpha \cdot v_2) + (v_2^2).
\end{equation*}
Since $L$ is nef and nef classes in $\HH^\infty$ are stable by pushforward, we have $(g_*L \cdot g_*L) \geqslant 0$, hence:
\begin{equation}\label{eq_self_intersection_2}
-(v_2^2) \leqslant 2 \deg(g) (\alpha \cdot v_2).
\end{equation}
We now compute $\deg(f\circ g)$:
\begin{equation} \label{eq_product_parabolic}
\deg(f \circ g) = (g^*f^* L \cdot L ) = (f^* L \cdot g_* L) = \deg(f) (\alpha \cdot v_2) + \deg(g) (\alpha \cdot v_1) + (v_1 \cdot v_2).
\end{equation} 
Since the intersection form is negative definite on $\{ v \in \HH^\infty | (v\cdot L)=0 \}$, the Cauchy-Schwarz inequality implies that:
\begin{equation*}
|(v_1 \cdot v_2)| \leqslant \sqrt{(v_1^2)(v_2^2)}.
\end{equation*}
Applying the above inequality to \eqref{eq_product_parabolic}, we get:
\begin{equation*}
\deg(f\circ g) \leqslant \deg(f) (\alpha \cdot v_2) + \deg(g) (\alpha\cdot v_1) + \sqrt{ (v_1^2)(v_2^2)}. 
\end{equation*}
We now apply \eqref{eq_self_intersection_1} and \eqref{eq_self_intersection_2}:
\begin{multline*}
\deg(f\circ g)  \leqslant \deg(f) (\alpha \cdot v_2) + \deg(g) (\alpha \cdot v_1)  
+ \sqrt{2\deg(g) (\alpha \cdot v_2) \left| 1 - 2 \deg(f) (\alpha \cdot v_1) \right|  }.
\end{multline*}
This last inequality together with \eqref{eq_normalization_1} and \eqref{eq_normalization_2} gives:
\begin{equation*}
\deg(f\circ g) \leqslant \deg(f) + \deg(g) + (2\deg(g))^{1/2}\sqrt{ \left| 1 - 2 {\deg(f)}  \right| }.
\end{equation*}
By Lemma  \ref{lem_lambda_deg}, we have $\sqrt{ 2 \deg(f) -1} \leqslant \sqrt{2 \deg(f)}$, hence:
\begin{equation*}
\deg(f\circ g) \leqslant  \left ( \sqrt{\deg(f) } + \sqrt{\deg(g)  }\right )^2.
\end{equation*}
We conclude that:
\begin{equation*}
\sqrt{{\deg(f\circ g)}} \leqslant  \sqrt{{\deg(f)}} + \sqrt{{\deg(g)}}  ,
\end{equation*}
as required.
\end{proof}

\begin{proof}[Proof of Proposition \ref{prop_kernel}]

$(ii) \Leftrightarrow (iii)$.
We claim that for any $f \in G$
\begin{equation} \label{E:trans}
\tau(\rho_f) = \log \lambda_1(f).
\end{equation}
This is because 
\begin{align*}
\tau(\rho_f) & = \lim_{n \to \infty} \frac{1}{n} d(\rho_{f^n}(L), L) \\
& = \lim \frac{1}{n} \cosh^{-1} \left( {\deg(f^n)} \right) \\
& = \lim \frac{1}{n} \left( \log \deg(f^n) + O(1) \right) \\
& = \log \lambda_1(f).
\end{align*}
where we used $\log x \leq \cosh^{-1}(x) \leq \log x + \log 2$.
Hence, the isometry $\rho_f$ is not loxodromic if and only if $\lambda_1(f) =1$.

We now prove the implication $(i) \Rightarrow (ii)$, by showing that if $f$ is loxodromic, then ${\lambda(f)} \neq 1$. 
If $f$ is loxodromic, there exists a nef class $\beta \in \partial \nefh$ fixed by $f$ and which is not proportional to $\alpha$. 
Suppose that $f^* \beta = \mu(f) \beta$ where $\mu \in \R^*$. By the Hodge index theorem, the product $(\alpha \cdot \beta)$ is non-zero and using the projection formula, we obtain:
\begin{equation*} 
 (\alpha \cdot \beta)=(f^* \alpha \cdot f^* \beta) = \lambda(f) \mu(f) (\alpha\cdot \beta).
\end{equation*} 
We thus obtain that $\mu(f) = \frac{1}{\lambda(f)}$. 
We now compute:
\begin{equation} \label{E:alpha-beta}
((f^n)^* (\alpha + \beta) \cdot L) = \left ( \lambda(f)^n (\alpha \cdot L) + \dfrac{1}{\lambda(f)^n} (\beta \cdot L)  \right )
\end{equation}
hence
\begin{equation*}
 ((f^n)^* (\alpha + \beta) \cdot L) = \left ({\lambda(f)}\right )^n (\alpha \cdot L) + \left (\dfrac{1}{\lambda(f)} \right )^n (\beta \cdot L).
\end{equation*}
Since $f$ is loxodromic and since $\alpha + \beta$ is big and nef, Lemma \ref{lem_big} shows that the above sequence must diverge to infinity, hence $\lambda(f) \neq 1$.

We finally show that $(ii) \Rightarrow (i)$. 
By contradiction, suppose that $\lambda(f) \neq 1$; then Lemma \ref{lem_lambda_deg}  implies
$$2 {\deg(f^n)} \geqslant \max \left ( \left ({\lambda(f)}\right )^n, \left ( \dfrac{1}{\lambda(f)} \right )^n  \right ).$$
We thus conclude that $\lim_{n\rightarrow +\infty}\frac{1}{n} \log {\deg(f^n)} = \tau(\rho_f) > 0$, hence $f$ is loxodromic, which contradicts our assumption. 
\end{proof}

\begin{proof}[Proof of Corollary \ref{cor_kernel}] 
If $f \in G$ is not loxodromic, then $\lambda_1(f) = \lambda(f)= 1$ by Proposition \ref{prop_kernel}.  Otherwise, $f$ is loxodromic, and by Lemma \ref{lem_big} and equation \eqref{E:alpha-beta} we have 
$$\begin{aligned}
\log \deg (f^n)  & =  \log\ ( (f^n)^*(\alpha + \beta) \cdot L) + O(1) \\
& =  \log\ ( \lambda(f)^n (\alpha \cdot L ) + \lambda(f)^{-n} (  \beta \cdot L)  ) + O(1)\\
\end{aligned}$$
Hence, $\lambda_1(f) = \max(\lambda(f) ,  \lambda(f)^{-1})$, as required.
\end{proof}

\section{General facts on random products}

\subsection{Random products of Cremona transformations}

We fix a countable semigroup $G$ of birational maps on $\Pg^2$ and consider a random walk with transition law $\mu$ on $G$. We assume that the support of the measure $\mu$ generates $G$ and that the following integral is finite:
\begin{equation} \label{E:first-moment}
\int_G \log \deg(g) \  d\mu(g) < +\infty.
\end{equation}
Recall that the algebraic degree is submultiplicative, i.e. 
$$ \deg(f \circ g) \leqslant\deg(f) \deg(g)$$
hence Kingman's subadditive ergodic theorem shows the existence of the limit 
\begin{equation*}
\ell_\mu := \lim_{n\rightarrow +\infty}{ \frac{1}{n}} \int_G  \log \deg(f) \ d\mu_n(f).
\end{equation*}

For our result, we shall need a random variable which is closely related to a normal law.
\begin{defi} Fix $\sigma>0$. The \emph{folded Gaussian} distribution parametrized by $\sigma$, denoted $\mathcal{FN}(0,\sigma)$ is the pushforward of the normal distribution $\mathcal{N}(0,\sigma)$ by the  map $\varphi(x) = |x|$.
\end{defi}

We will then apply the following consequence of the central limit theorem.

\begin{prop} \label{prop_disjunction}
Consider a sequence $(Z_n)$ of $i.i.d.$ variables of mean $m$ and of variance $\sigma$.
Then the following holds: 
\begin{enumerate}
\item If $m \neq 0$, then 
$$\frac{|\sum_{i = 1}^n Z_i | - n |m|}{\sqrt{n}} \to \mathcal{N}(0, \sigma).$$
\item 
If $m =  0$, then 
$$\frac{|\sum_{i = 1}^n Z_i | }{\sqrt{n}} \to \mathcal{FN}(0, \sigma)$$
where 
$\mathcal{FN}(0, \sigma)$ is the folded  normal distribution.
\end{enumerate}
\end{prop}

\section{Proof of Theorem A for elementary semigroups}

\subsection{Central limit theorem for the algebraic degree}

\subsubsection{Degree for elliptic semigroups}

Suppose that $G$ induces an elliptic semigroup action on $\nefh$.
Observe that we can write
\begin{equation*}
\log \deg(f_n) =   \log { (f_n^*L \cdot L)}.  
\end{equation*}
Since the semigroup $G$ is elliptic and $d(L, \rho_{f_n} (L)) = \log { (f_n^*L \cdot L)} + O (1)$ by Lemma \ref{lem_distance_log}, the second term above is bounded, so the sequence
\begin{equation*}
\dfrac{\log \deg(f_n) - n \ell_\mu}{\sqrt{n}} 
\end{equation*}
converges to the Dirac mass.

\subsubsection{Lineal semigroups}

Suppose the action of $G$ on the hyperbolic space is lineal. Let $\theta_+ , \theta_-$, be the two invariant nef classes on the 
boundary $\partial \nefh$ which are either globally $G$-invariant by pullback or pushforward or swapped, normalized so that $(\theta_+ \cdot L) = (\theta_- \cdot L) = 1$. 
Let $\lambda : G \to \mathbb{R}$ be defined so that $f^*(\theta_+) = \lambda(f) \theta_+$. Since the classes $\theta_{+}$ and $\theta_-$ are invariant classes on the boundary, the random variables 
$$\log  (f_n^*\theta_{+} \cdot L)\qquad \textup{and} \qquad \log (f_n^*\theta_{-} \cdot L)$$ 
describe a random walk on the real line by Lemma \ref{trick_lemma}.  
We prove the following central limit theorem. 

\begin{thm} \label{T:CLT-lineal}
Suppose that the averages given by:
\begin{equation*}
 \Lambda_\mu := \int_G \log \lambda(f) \ d\mu(f)
\end{equation*}
are finite, and moreover the variance 
\begin{equation*}
 \sigma^2 := \int_G \left( \log \lambda(f) - \Lambda_\mu\right)^2 \ d\mu(f)
\end{equation*}
is also finite. Then: 
\begin{enumerate}
\item If $\Lambda_\mu \neq 0 $, then 
$$\frac{\log \deg(f_n) - n  \Lambda_\mu}{\sqrt{n}} \to \mathcal{N}(0, \sigma).$$
\item 
If $  \Lambda_\mu = 0$, then 
$$\frac{\log \deg(f_n)  }{\sqrt{n}} \to \mathcal{FN}(0, \sigma)$$
where 
$\mathcal{FN}(0, \sigma)$ is the absolute value of the normal law $\mathcal{N}(0, \sigma)$.
\end{enumerate}
\end{thm}

To prove Theorem \ref{T:CLT-lineal}, we first compute the degree as follows.

\begin{lem} \label{L:lineal-degree}
For each $f \in G$, we have 
$$\log \deg(f) =   \left| \log \lambda(f)  \right| + O(1)$$
where $O(1)$ is a  constant which depends only on the intersection product $(\theta_+\cdot \theta_-)$. 
\end{lem}

\begin{proof}
Take $\theta_+ , \theta_-$ the two invariant nef classes on the boundary which are either globally $G$-invariant or swapped, 
normalized so that $(\theta_+ \cdot L) = (\theta_- \cdot L) = 1$.
Observe that the Hodge index theorem implies that $\theta_+ + \theta_-$ is big and nef. By Lemma \ref{lem_big}, we have:
\begin{equation} \label{eq_lineal}
|\log \deg(f) - \log \left ( (f^*\theta_+ \cdot L) + (f^* \theta_- \cdot L) \right )| \leqslant C,
\end{equation} 
where $C$ depends only on $(\theta_+ \cdot \theta_-)$.
Since $f^* \theta_+ = \lambda(f) \theta_+$, $f_* \theta_- = \lambda(f)\theta_-$,  
%and using Siu's inequality $L \leqslant C (\theta_+ + \theta_-)$, 
we have $f^* \theta_- = \frac{1}{\lambda(f)} \theta_-$ and 
\begin{equation} 
\log \left ( (f^*\theta_+ \cdot L) + (f^* \theta_- \cdot L) \right ) = \log \left ( \lambda(f) + \frac{1}{\lambda(f)} \right) =  \log (2 \cosh u(f) )
\end{equation} 
where $u(f) = \log \lambda(f)$ and $\ch$ is the hyperbolic cosine.
Hence, we rewrite \eqref{eq_lineal}  
as follows:
\begin{equation} \label{eq_lineal_ch}
\left |\log {\deg(f) }  - \log \ch u(f) \right | \leqslant C'
\end{equation}
with $C' = C + \log 2$.
Observe that the following inequality is satisfied for all $x \in \mathbb{R}$:
\begin{equation*}
\dfrac{e^{|x|}}{2} \leqslant \ch(x) \leqslant e^{|x|},
\end{equation*}
hence we get:
\begin{equation*}
\left | \log \ch u(f) - |u(f)| \right | \leqslant \log 2.
\end{equation*}
In particular, using the above equation and \eqref{eq_lineal_ch},  we obtain:
\begin{equation} \label{eq_log_u}
\left |\log {\deg(f)}   - |u(f)| \right | \leqslant C'',
\end{equation}
where $C'' = C + 2 \log 2> 0$. 
Hence, we decompose $\log \deg(f) $ as follows:
\begin{equation} \label{eq_decomp_deg}
\begin{aligned}
\log \deg(f) &  =\left| \log \lambda(f) \right| + O(1),
\end{aligned}
\end{equation}
completing the proof of the lemma.
\end{proof}

\begin{proof}[Proof of Theorem \ref{T:CLT-lineal}]

We have
\begin{equation*}
\log \deg(f_n) = |\log \lambda(f_n)| = \left| \sum_{i=1}^n \log(g_i) \right|.
\end{equation*}
Using Proposition \ref{prop_disjunction}, we conclude that Theorem \ref{T:CLT-lineal} holds.
\end{proof}

\subsubsection{Parabolic semigroups}

We now suppose that the action of $G$ on $\HH^\infty$ is parabolic.
Take $\alpha \in \partial \HH^\infty \cap \partial \nefh$ a globally $G$-invariant class, we shall choose $\alpha $ so that $(\alpha \cdot L) =1$.
For each $f\in G$, we have $f^*\alpha =\alpha$.
Note that $\lambda_1(f) = 1$ and by hypothesis 
%\begin{equation} \label{E:finite}
%\int_G {\deg(f)} \ d\mu(f) < +\infty
%\end{equation}
%holds, we also have 
\begin{equation} \label{eq_assumption_parabolic_focal}
\int_G \sqrt{\deg(f)} \ d\mu(f) < +\infty.
\end{equation}

By Lemma \ref{lem_cantat_blanc}, we have:
\begin{equation} \label{eq_focal_cocycle_condition}
\sqrt{\deg(f\circ g)} \leqslant \sqrt{\deg(f)} + \sqrt{\deg(g)}.
\end{equation}
Moreover, equation \eqref{eq_assumption_parabolic_focal} proves that the cocycle $\sqrt{\deg (\cdot)}$ belongs to $\LL^1(\mu)$, 
hence Kingman's subadditive ergodic theorem yields the almost sure convergence  
\begin{equation*}
\lim_{n\rightarrow +\infty} \dfrac{1}{n} \sqrt{ {\deg(f_n)}} = C,
\end{equation*}
where $C \in [0, \infty)$. 
This proves that almost surely 
$$\limsup_{n \to \infty} \dfrac{1}{\sqrt{n}} \log {\deg(f_n)} \leqslant 0.$$
Note also that
$$\liminf_{n \to \infty} \dfrac{1}{\sqrt{n}} \log {\deg(f_n)} \geqslant 0,$$
hence the sequence of random variables 
$$\dfrac{1}{\sqrt{n}} ( \log \deg(f_n))$$ 
converges to zero almost surely, hence in probability 
and the central limit theorem also holds for $\log \deg(f_n)$.  

\subsection{Central limit theorem for the first dynamical degree}

We now prove that $\log\lambda_1(f_n)$ satisfies a central limit theorem. 
Since there is an invariant class on the boundary, Corollary \ref{cor_kernel} holds and $\lambda_1(f) = \max(\lambda(f), \lambda(f)^{-1})$ for all $f$ in $G$. 
This proves that 
$$\log \lambda_1(f_n)  =   \left |\log {\lambda(f_n)}\right |. $$

Thus, as in the proof of Theorem \ref{T:CLT-lineal}, we obtain that the sequence
\begin{equation}
\dfrac{\log \lambda_1(f_n) - n \ell_\mu}{\sqrt{n}} 
\end{equation} 
 converges to 
\begin{equation*}
\mathcal{N}(0 , \sigma) 
\end{equation*}
if $\Lambda_\mu \neq 0$ and to 
\begin{equation*}
\mathcal{FN}(0, \sigma)
\end{equation*}
if $\Lambda_\mu = 0$.
This completes the proof of Theorem A for elementary semigroups.

\subsection{Summary of the Central limit theorem in the elementary case} \label{section_summary_elementary}
Let us set
\begin{equation*}
\Lambda_\mu = \int \log \lambda(f) \ d\mu(f),
\end{equation*}
where $f^*\alpha = \lambda(f) \alpha$ for all $f\in G$, with $\alpha \in \partial \HH^\infty \cap \nefh$.
The following table summarizes all possible limit behaviours.
Note that, as a consequence:

\begin{cor} 
If the semigroup $G$ is elementary, then $\sigma = 0$ if and only if $G$ is elliptic, parabolic, or if $G$ is lineal 
with $\lambda(f)$ constant on the support of $\mu$. 
\end{cor}

%, in the lineal case one has $\sigma = 0$ if and only if $\lambda(f)$ is constant on the support of $\mu$. 

\medskip

{\small
\begin{tabular}{|l|l|l|l|l|}
\hline
\textbf{Type of group} & \textbf{Mean} & \textbf{Limit law for $\log \deg$} & \textbf{Limit law for $\log \lambda_1$} \\
\hline
elliptic & $\ell = 0$ & Dirac mass at zero & Dirac mass at zero \\
%& 	& $\sigma =0$ & $\sigma^2 = 0$ \\
\hline
parabolic &$\ell =0$ &Dirac mass at zero & Dirac mass at zero\\
\hline
lineal, $  \Lambda_\mu = 0 $ &  $\ell=0$ & Folded Gaussian & Folded Gaussian \\
 &  & $\sigma^2 = \int (\log \lambda(f))^2 d\mu$ & $\sigma^2 = \int (\log \lambda(f))^2 d\mu$ \\
\hline
lineal, $ \Lambda_\mu<0$ & $\ell =- \Lambda_\mu >0$ & Gaussian & Gaussian \\
 &  &  $\sigma^2 = \int \left( - \log {\lambda(f)}  + \Lambda_\mu \right)^2 d\mu$ & $\sigma^2 = \int \left(- \log {\lambda(f)}  + \Lambda_\mu \right)^2 d\mu$ \\
\hline
lineal, $  \Lambda_\mu>0$ & $\ell = \Lambda_\mu>0$  & Gaussian & Gaussian \\
 &  & $\sigma^2 = \int \left( \log \lambda(f) - \Lambda_\mu \right)^2 d \mu$ & $\sigma^2 = \int \left( \log \lambda(f) - \Lambda_\mu \right)^2 d \mu$ \\
\hline
non-elementary &  &  & \\ 
(see next sections) &$\ell > 0$ &  Gaussian & Gaussian \\
\hline
\end{tabular}
}

\section{Non-elementary semigroups}

Let us now assume that the semigroup $G$ generated by the support of $\mu$ is non-elementary. 
We recall the following results due to Maher-Tiozzo. 

\begin{thm}[Maher-Tiozzo \cite{maher_tiozzo}] 
Let $G$ be a non-elementary, countable semigroup of isometries of a $\delta$-hyperbolic space $X$, with hyperbolic (Gromov) boundary $\partial X$.
Let $\mu$ be a measure whose support generates $G$, and let $o \in X$ be a base point. 
Then for almost every sample path $f_n = g_1 \cdot \ldots \cdot g_n$, the sequence $(f_n o)$ converges to a point $\xi \in \partial X$. 
Moreover, the resulting hitting measure  is non-atomic and is the unique $\mu$-stationary measure on the boundary. 
\end{thm}

Recall that a measure $\nu$ on a $G$-space $M$ is \emph{$\mu$-stationary} if $\int_G g_\star \nu \ d\mu(g) = \nu$. Moreover, it is 
\emph{$\mu$-ergodic} if it is not a non-trivial convex combination of $\mu$-stationary probability measures on $M$.

\begin{thm}[Maher-Tiozzo \cite{maher_tiozzo_cremona}]
Let $\mu$ be an atomic non-elementary probability measure on $\Bir(\Pg^2)$ with finite first moment. Then there exists $\ell_\mu > 0$ such that for a.e. random product $f_n = g_1 \cdot \ldots \cdot g_n$, we have:
\begin{equation*}
\lim_{n\rightarrow +\infty} \dfrac{1}{n} \log\deg(f_n) = \ell_\mu.
\end{equation*}
Moreover, if $\deg(f)$ is bounded on the support of $\mu$, then for almost every sample path, one has:
\begin{equation*}
\lim_{n\rightarrow +\infty } 
\dfrac{1}{n}\log \lambda_1(f_n) = \ell_\mu.
\end{equation*}
\end{thm}

\subsection{The horofunction boundary} 
Let us recall the construction of the horofunction compactification of a non-proper hyperbolic space, as developed in \cite{maher_tiozzo}. 

Let $(X, d)$ be a metric space and let $o \in X$ be a base point. Then we define for each $x \in X$ the map $\rho_x : X \to \mathbb{R}$
$$\rho_x(z) := d(x, z) - d(x, o)\qquad \textup{for }z \in X.$$
The function $\rho_x$ is $1$-Lipschitz, and $\rho_x(o) = 0$. The assignment $x \mapsto \rho_x$ defines a map $\Phi : X \to \textup{Lip}^1(X)$
into the space of $1$-Lipschitz functions on $X$. The \emph{horofunction compactification} $\overline{X}^h$ of $X$ is defined as the closure of $\Phi(X)$ in $\textup{Lip}^1(X)$, with respect to the topology of 
pointwise convergence. If $X$ is separable, then $\overline{X}^h$ is compact and metrizable. 
Elements of $\overline{X}^h$ are called \emph{horofunctions}, and there are two types of them: \emph{finite horofunctions}, if $\inf_{x\in X} h(x) \in \mathbb{R}$, and \emph{infinite horofunctions} if $\inf_{x \in X} h(x) = - \infty$. We denote as $X_\infty^h$ the space of infinite horofunctions.

Moreover, there is a \emph{local minimum map} $\pi : \overline{X}_\infty^h \to X \cup \partial X$ defined as follows. 
If $h \in X_\infty^h$, then there exists a sequence $(x_n) \subseteq X$ such that $h(x_n) \to - \infty$. It turns out that 
such a sequence must converge in the Gromov topology to a point in the Gromov boundary $\partial X$, and the limit point 
does not depend on the particular choice of $(x_n)$. Hence, one defines a $G$-equivariant map $\pi : X_\infty^h \to \partial X$ as 
$$\pi(h_x) := \lim_{n \to \infty} x_n \in \partial X.$$
In fact, the local minimum map can also be defined for finite horofunctions, but we do not need it here.
By \cite[Proposition 4.4]{maher_tiozzo}, any $\mu$-stationary probability measure $\nu$ on $\overline{X}^h$ only charges infinite horofunctions, i.e. $\nu(X_\infty^h) = 1$.

\subsection{Central limit theorems for cocycles}

Fix $G$ a semigroup of birational maps of $\Pg^2$, and let $M$ be a compact $G$-space. 
Recall that a \emph{cocycle} is a function $\sigma : G \times M \to \mathbb{R}$ such that 
$$\sigma(gh, x) = \sigma(g, h x) + \sigma(h, x)\qquad \qquad \forall g, h \in G, \forall x \in M.$$
A cocycle $\sigma: G \times M \to \mathbb{R}$ has \emph{constant drift} $\lambda$ if there exists $\lambda \in \mathbb{R}$ such that
$$\int_G \sigma(g, x) \ d\mu(g) = \lambda$$
for any $x \in M$. 
A cocycle $\sigma : G \times M \to \mathbb{R}$ is \emph{centerable} if it can be written as
\begin{equation*}
\sigma(g, x) = \sigma_0(g, x) + \psi(x) - \psi(g \cdot x)
\end{equation*}
where $\sigma_0$ is a cocycle with constant drift and where $\psi : M \to \mathbb{R}$ is a bounded, measurable function.
Given a cocycle, we denote by $\sigma_{sup}(g) := \sup_{x\in M} | \sigma(g,x)|$. Finally, a cocycle has \emph{unique covariance} $v$
if 
$$v^2 = \int_{G \times M} (\sigma(g, x) - \lambda)^2 \ d\mu(g) d\nu(x)$$ 
for any $\mu$-stationary measure $\nu$. Recall the key ingredient in Benoist-Quint's central limit theorem for cocycles (\cite[Theorem 3.4]{benoist_quint_central}).

\begin{thm}[Central limit theorem for cocycles, I]  \label{T:CLT-cocycle}
Let $G$ be a discrete group, $M$ be a compact metrizable $G$-space and $\mu$ an atomic measure on $G$. Assume 
$\sigma : G \times M \to \R$ is a centerable cocycle with drift $\lambda$ and unique covariance $v \geqslant 0$ and such that
$$ \int_G \sigma_{sup}^2(g)  \ d\mu(g)< +\infty.$$ 
Then for any bounded continuous function $F$ on $\mathbb{R}$, uniformly in $x\in M$, one has:
\begin{equation*}
\lim_{n\rightarrow +\infty} \int_{G} F \left ( \dfrac{\sigma(g,x) - n \lambda}{\sqrt{n}} \right ) d\mu_n(g) = 
\frac{1}{\sqrt{2 \pi} v} \int_{\R} F(t) e^{-\frac{t^2}{2v^2} } \ dt.  
\end{equation*}
\end{thm}

However, a more general version of this theorem does not require the cocycle to have unique covariance. Indeed we have the following. 
As remarked in \cite[Remark 1.7]{horbez}, the proof is exactly the same as the proof of \cite[Theorem 4.7]{benoist_quint_hyperbolic}.

\begin{thm}[Central limit theorem for cocycles, II] \label{sigma-CLT}
Let $G$ be a discrete group, $M$ be a compact metrizable $G$-space and $\mu$ an atomic measure on $G$.
Let $\nu$ be a $\mu$-ergodic, $\mu$-stationary probability measure on $M$, and let $\sigma : G \times M \to \R$ be a 
centerable cocycle with drift $\lambda$. Then there exists $v \geqslant 0$ such that
for $\nu$-a.e. $x \in M$ we have, for any bounded, continuous function $F$, 
$$\lim_{n \to \infty} \int_G F\left( \frac{\sigma(g, x) - n \lambda}{\sqrt{n}} \right) \ d \mu_n(g) = 
\frac{1}{\sqrt{2 \pi} v} \int_{\R} F(t) e^{-\frac{t^2}{2v^2} } \ dt.$$
\end{thm}

\begin{rem}
Let us note that \cite{maher_tiozzo} define the random walk as $f_n = g_1\dots g_n$, while \cite{benoist_quint_hyperbolic}, \cite{horbez} use the 
definition  $f_n = g_n \dots g_1$. 
In this paper, we define the random walk as $f_n = g_1 \dots g_n$ on the semigroup of rational maps, which, since the pullback is contravariant,
induces the random walk $\rho_{f_n} = \rho_{g_n} \dots \rho_{g_1}$ on the space of isometries. Thus, we can use the results of 
\cite{benoist_quint_hyperbolic}, \cite{horbez} verbatim. 
Note finally that the $n$-step distributions $\mu_n$ of the left and right random walk are equal, hence, as far as convergence in probability is concerned, 
results on one and the other are equivalent. On the other hand, results on almost sure convergence do \emph{not} automatically translate, but we do not 
directly use them here. 
\end{rem}

\subsection{The Busemann cocycle}

Let us now define $V$ as the subset of the Picard-Manin space given by $V := \overline{\textup{Span}_{f \in G}(f^\star L)}$ and $X := V \cap \mathbb{H}^\infty$.
Then $X$, with the metric $d$ induced by $\mathbb{H}^\infty$, is a geodesic, $\delta$-hyperbolic, and separable (since $G$ is countable) metric space, hence 
we can construct its horofunction compactification $M := \overline{X}^h$, which is metrizable. Moreover, $G$ acts by isometries on $X$ and by homeomorphisms on $M$.

Let us define the \textbf{Busemann cocycle}, denoted  $\beta : G \times  \overline{X}^h \to \mathbb{R} $, as
\begin{equation*}
\beta(g, x) := h_x( g_* \cdot L),
\end{equation*}
where $h_x$ is the horofunction associated to $x$. 
We use the following properties of the Busemann cocycle. 

\begin{prop} \label{P:cocycle-prop}
Let $\beta : G \times \overline{X}^h \to \mathbb{R}$ be the Busemann cocycle, and let $\mu$ be an atomic probability measure on the group of isometries of $(X,d)$ 
with finite second moment. Then there exists $\lambda \in \mathbb{R}$ such that:
\begin{enumerate}
\item \textup{(\cite[Corollary 2.7]{horbez})}
For any $\mu$-stationary measure $\nu$ on $\overline{X}^h$, 
$$\int \beta(g, x) \ d\mu(g) d\nu(x) = \lambda.$$
\item \textup{(\cite[Proposition 2.8]{horbez})} 
For all $\epsilon > 0$ there exists a sequence $(C_n) \in \ell^1(\mathbb{N})$ such that 
$$\mu_n( g \in G \ : \ |\beta(g, x) - n \lambda| \geqslant \epsilon n ) \leqslant C_n$$
for any $x \in \overline{X}^h$.
\end{enumerate}
\end{prop}

Let us recall that  there is a $G$-equivariant map $\pi : X^h_\infty \to \partial X$ from the set of infinite horofunctions to the Gromov boundary. We shall exploit the following result \cite[Corollary 2.3]{horbez}.

\begin{lem} \label{L:cocycle-compare}
For all $x, y \in X^h_\infty$ such that $\pi(x) \neq \pi(y)$, there exists $C > 0$ such that for all $g \in G$
$$d(o, go)  - C \leqslant \max \{ \beta(g, x), \beta(g, y) \} \leqslant d(o, go).$$
\end{lem}

Recall that the \emph{Gromov product} between $y$ and $z$ based at $x$ is $\langle y, z \rangle_x := \frac{d(x, y) + d(x, z)  - d(y, z)}{2}$.
We use the following basic fact about the Gromov product (see e.g. \cite[Proposition 5.8]{maher_tiozzo}). 

\begin{lem} \label{lem_inequality_gromov}
Let $(X, d)$ be a $\delta$-hyperbolic space, and let $o \in X$ be a base point. 
Then there exists a constant $C>0$ such that, for any isometry $f$ of $X$, 
\begin{equation*}
|\tau(f) - d(o , f o ) + 2 \langle f o, f^{-1} o \rangle_o | \leqslant  C.
\end{equation*}
\end{lem}

Moreover, we use the fact that the Gromov product decays faster than any given function:

\begin{lem}[Taylor-Tiozzo \cite{taylor-tiozzo}, Lemma 3.4] \label{L:small-gromov}
Let $\mu$ be a non-elementary probability measure on a countable group $G$ of isometries of a $\delta$-hyperbolic space $X$,
let $o \in X$ be a base point and let $(f_n)$ be a random walk driven by $\mu$. 
Then for any function $\varphi : \mathbb{N} \to \mathbb{R}$ with $\limsup_{n \to \infty} \frac{\varphi(n)}{n} = 0$, we have 
$$\mathbb{P}\left( \langle f_n o, f_n^{-1} o \rangle_o \geqslant \varphi(n) \right) \to 0.$$
\end{lem}

\section{Arithmeticity properties of the length spectrum} \label{S:arithmetic}

\begin{defi}
We call the \emph{length spectrum} of a semigroup $G < \textup{Bir}(\mathbb{P}^2)$ the set 
$$LS(G) := \{ \log \lambda_1(g) \ : \ g \in G \}.$$
Then, we say $G$ has \emph{arithmetic length spectrum} if there exists $a \in \mathbb{R}$ 
such that $LS(G) \subseteq a \mathbb{N}$.
Otherwise, we say the length spectrum of $G$ is \emph{non-arithmetic}. 
\end{defi}

Let us note for discrete subgroups of the group of isometries of a finite dimensional hyperbolic space, the length spectrum can never be arithmetic unless the group is elementary, as shown by \cite{Dalbo} in dimension $2$ and \cite{kim-arith} in any dimension. 
We will now show, however, that in infinite dimension, as in the case of the Cremona group, there exist non-elementary subgroups
with arithmetic length spectrum, proving Proposition \ref{P:exist-arith} from the introduction.

\begin{prop} \label{P:ex-ar}
There exist non-elementary subgroups of $\Bir(\Pg^2)$ which have arithmetic length spectrum.
\end{prop}

The class of examples we construct is as follows. Consider two polynomials $P_1, P_2 \in \C[x]$ of degree $d_1, d_2$ such that $\deg(P_1 \pm P_2) = \max (\deg(P_1), \deg(P_2))$. Take two elementary maps $e_i := (x,y) \mapsto ( x + P_i(y), y) $ where $i=1,2$ and take $a := (x,y) \mapsto (2 x + y, x+ y)$ the cat map. 
Consider $F := e_1 \circ a$ and $G := a \circ e_2 \circ a^2$.     

%We postpone the proof to the end of the section and 
We begin with the following observation.
\begin{lem} \label{L:length}
The length spectrum of the group $\Gamma :=  \langle F,G \rangle$ is  $\mathbb{N}\log(d_1) + \mathbb{N} \log(d_2).$
\end{lem}

\begin{proof}
Denote by $E$ the subgroup of elements of the form $(\alpha x + P(y), \beta y + \gamma )$ where $\alpha ,\beta\in\C^*,\gamma \in \C$ and let us denote by $A$ the group of affine transformations. 
Any non-trivial element $g$ of $\Gamma$ can be conjugated to $g' = F^{i_1} G^{j_1} \ldots F^{i_n} G^{j_n}$, where $i_1,j_n \in \mathbb{Z}$, $ i_2 ,\ldots,i_n, j_1 ,\ldots, j_{n-1} \in \mathbb{Z}^* $ and so that the word is cyclically reduced (i.e., the last letter is not the inverse of the first one).
Since $a \notin A\cap E$, this element  can be decomposed into an alternating product of   elements in  $\{ e_1, e_2, e_1^{-1}, e_2^{-1} , e_2  e_1, e_1^{-1} e_2^{-1}\} \subset E \setminus (A\cap E)$ and elements in $\{ a^k \}_{k\in \mathbb{Z}^*} \subset A \setminus (A\cap E)$. 
Using \cite[Theorem 2.1]{friedland_milnor}, we deduce that the degree of $g'$ is of the form $ d_1^k d_2^l$ where $k,l \in \mathbb{N}$. Moreover, since $g'$ is cyclically reduced, $\deg((g')^n) = (\deg(g'))^n$ for any $n \geq 1$; hence, $\lambda_1(g) = \lambda_1(g')$ is also of the form $ d_1^k d_2^l$, as required.
\end{proof}

In the remaining part of the section, we will show that the action of the group $\langle F,G\rangle$ on the hyperbolic space $\mathbb{H}^\infty$
is non-elementary. 

We shall introduce a particular class of $b$-divisors (i.e. elements of $\varprojlim \NS(X)$), associated to valuations.  
A valuation $\nu $ on $\C^2$ is a function from the field of rational functions $\C(x,y)$ to $\R \cup \{+\infty \}$ satisfying the following conditions: 
\begin{enumerate}
\item[(i)] $\forall f,g\in \C(x,y)$, $\nu(fg) = \nu(f) + \nu(g)$. 
\item[(ii)] $\nu(0)= +\infty$.
\item[(iii)] $\nu(f+g) \geq \min(\nu(f), \nu(g))$ for all $f,g \in \C(x,y)$. 
\end{enumerate} 

A basic example is the valuation $\nu_0=-\deg$. 
We denote by $\mathcal{V}_0$ the set of valuations on $\C^2$ such that there exists a generic affine function $L$ satisfying $\nu(L)<0$. It is endowed with the topology of pointwise convergence.   
One important application is that from any sequence of valuations $(\nu_n)$ which is pointwise bounded, we can extract a subsequence converging to a valuation $\nu_\infty$. Any polynomial automorphism $F$ acts on valuations as follows: $(F_* \nu)(f) := \nu(f \circ F)$.

To any appropriate valuation $\nu$, one can associate a class in the Picard-Manin space $Z_\nu$. We use the following results of Favre-Jonsson:

\begin{lem}[\cite{favre_jonsson_dynamical_compactification}, Lemma A.1] The assignment $\nu \to Z_\nu$ from $\mathcal{V}_0$ to the set of $b$-divisors over $\mathbb{P}^2$ is a continuous injection. 
\end{lem}

\begin{lem}[\cite{favre_jonsson_dynamical_compactification}, Lemma A.6] \label{L:functorial}
For any $\nu \in \mathcal{V}_0$ and any automorphism $F$ of $\mathbb{C}^2$, one has $Z_{F_* \nu} = F_* Z_\nu  = (F^{-1})^* Z_\nu$. 
\end{lem}

A particular property of valuations is that they carry some geometric meaning. Take $\nu \in \mathcal{V}_0$, we choose some affine coordinates at infinity $(u,v)$ in $\mathbb{P}^2$ such that $\nu(u),\nu(v) \geq 0$. 
The center of a valuation $\nu $ in $\mathbb{P}^2$, denoted $C(\nu)$, is the Zariski closure in $\mathbb{P}^2$ of the  locus of points $p$ for which any polynomial $P \in \C[u,v]$ such that $\nu(P)>0$ must vanish at $P$.   
In our setting, the center in $\mathbb{P}^2$ of a valuation   is an irreducible subvariety of $\mathbb{P}^2$ (see e.g \cite[\S 2.1 p. 24]{vaquie_unif}); in particular, it may be either a point or an irreducible curve.

\begin{ex} The center of the valuation $\nu_0= -\deg$ is the line at infinity. 
To show this, one chooses coordinates $(u,v) \mapsto [1: u: v]$ so that $v=0$ corresponds to the line at infinity. We have $\nu_0(v) =1$ and $\nu_0(u)=0$, and $\nu(\C[u,v])\geq 0$. One sees that a polynomial $P \in \C[u,v]$ for which $\nu(P)>0$ is of the form $v^k \tilde{P}$, where $\tilde{P} \in \C[u,v]$, $k >0$ and $v \nmid \tilde{P}$. This shows that any point belongs to the center $C(v)$ if and only if it lies on the line at infinity $v=0$.  
\end{ex}

We start with the following observation. 

\begin{lem} \label{L:center}
Assume $F$ is any polynomial automorphism of $\C^2$ which contracts the line at infinity to a point $p$ at infinity. Then $F_* \nu_0$ is centered at the point $p$. 
\end{lem}

\begin{proof}
Choose some affine coordinate $(u,v)$ near $p$ such that $v=0$ is the equation of the line at infinity. 
For any polynomial  $f\in \C[u,v]$ which vanishes at $p$, the function $f\circ F$ is of the form: 
\begin{equation}
 f\circ F = v^k  g,
 \end{equation} 
 where $k \in \mathbb{N}^*$, $g$ is a rational function which does not vanish on the locus $v=0$. 
 In particular, $\nu_0 (f\circ F)=k >0$. 
 Moreover, if $f$ does not vanish at $p$, then the integer $k$ will be zero, so this shows that any other point distinct from $p$ on the line at infinity is not in the center of $F_* \nu_0$. This proves that the center of $F_* \nu_0$ is reduced to the point $p$. 
\end{proof}

In the following, we will need the fact that $ Z_{\nu_0}= L$, and in particular $Z_{\nu_0} \in \mathbb{H}^\infty$. 
\begin{prop} \label{L:non-elt}
The group $\Gamma = \langle  F,G\rangle$ has a non-elementary action on $\mathbb{H}^\infty$.
\end{prop}

\begin{proof}
Using the fact that $\lambda_1(F), \lambda_1(G)>1$, the induced transformations are loxodromic and there exist two classes $\theta_F, \theta_G$ on the boundary: 
$$ \lim \dfrac{1}{d_1^n}(F^{-n})^* L  = \theta_F $$
 $$ \lim \dfrac{1}{d_2^n}(G^{-n})^* L  = \theta_G. $$
We show that $\theta_F$ and $\theta_G$ are distinct. Using Lemma \ref{L:functorial} and $ Z_{\nu_0}= L$, we have: 
$$Z_{F^n_* \nu_0} = F^n_* L = (F^{-n})^* L$$
$$Z_{G^n_* \nu_0} = G^n_* L = (G^{-n})^* L.$$
We argue that the sequences of valuations $\nu_{n, F} := \frac{1}{d_1^n} F^n_* \nu_0$ and  $\nu_{n, G}:=  \frac{1}{d_2^n} G^n_* \nu_0$ diverge to different points of the valuative tree. 

Indeed, we first observe that $\deg(F^n) = d_1^n$, so that $(F^n)_* \nu_0 (x) \geq - d_1^n$ and $(F^n)_* \nu_0 (y) \geq - d_1^n$, so for any polynomial $P \in \C[x,y]$ we have: 
\begin{equation}
0 \geq  \dfrac{1}{d_1^n} (F^n)_* \nu_0 (P) \geq - \deg(P), 
 \end{equation} 
hence the sequence $(\nu_{n, F})$ is pointwise bounded and we can extract a subsequence converging to some $\nu_{\infty,F}$.
Similarly, we can extract a subsequence of $(\nu_{n, G})$ which converges to $\nu_{\infty,G}$.  

 If $\nu_{\infty, F} = \nu_{\infty,G}$ up to scaling, then the limiting valuations would be centered at the same point in $\mathbb{P}^2$. 
Note that $F^n$ contracts the line at infinity to the point $[1:0:0]$, so by Lemma \ref{L:center} the valuation $\nu_{\infty, F}$ is centered at $[1:0:0]$ in $\mathbb{P}^2$, whereas the valuation $\nu_{\infty,G}$ is centered at the point $a([1:0:0])= [2:1:0]$. Hence $\nu_{\infty, F}$ and $\nu_{\infty,G}$ are distinct, as required. 
The same argument shows that the valuations $\nu_{\infty,F^{-1}},\nu_{\infty,G^{-1}} $ are centered at $[0:1:0]$ and $[-1:2:0]$ respectively.
Since all the points $[1:0:0],[0:1:0], [2:1:0], [-1:2:0]$ are mutually distinct, we deduce that the semigroup generated by $F,G$ is non-elementary.
\end{proof}

\begin{proof}[Proof of Propositions \ref{P:ex-ar} and \ref{P:exist-arith}]
Let us consider two polynomials $P_1(x), P_2(x)$ in one variable of the same degree $d \geq 2$, 
so that both $P_1(x) + P_2(x)$ and $P_1(x) - P_2(x)$ still have degree $d$. Then if we define $F, G$ as above, by Lemma \ref{L:length} the semigroup $\Gamma$ 
generated by $F, G$ has arithmetic length spectrum, contained in $\log(d) \mathbb{N}$, and by Lemma \ref{L:non-elt} its action on $\mathbb{H}^\infty$ is non-elementary. 
Moreover, for any word $w = F^{i_1} G^{j_1} \dots F^{i_k} G^{j_k}$ in $F, G$ of length $n$, we have 
$$\log (\deg w) = \log(d) \sum_{\ell = 1}^k (i_\ell + j_\ell) = n \log(d),$$ 
which shows that \eqref{E:tau-arith} holds, hence $\sigma = 0$.
\end{proof}

\subsection{Proof of Theorem A for non-elementary subgroups}

We are now ready to complete the proof of our main theorem.

\begin{proof}[Proof of Theorem A; non-elementary case]
Let us first prove the CLT for $\log \deg (f_n)$. 
By \cite[Proposition 1.5]{horbez} the cocycle $\beta(g, x)$ is centerable, as a consequence of Proposition \ref{P:cocycle-prop}. 
Now, since $M$ is compact, there exists a $\mu$-stationary measure $\nu$ on $M$, and by taking its ergodic components
we can assume that $\nu$ is $\mu$-ergodic. By \cite[Proposition 4.4]{maher_tiozzo}, stationarity implies $\nu(\overline{X}^h_\infty) = 1$. 

Moreover, by Theorem \ref{sigma-CLT} we obtain that for $\nu$-almost every $x \in M$, 
the sequence $(\beta(f_n, x))$ satisfies a central limit theorem. Then, by taking a generic $x \in \overline{X}^h_\infty$ and applying 
Lemma \ref{L:cocycle-compare}, we obtain a central limit theorem for 
$$d(L, \rho_{f_n}(L)).$$
We have by definition
$$\log \deg(f_n) =  \log \left  ( {f_n^* L} \cdot L \right )$$
thus, since $$ \left  ( {f_n^* L} \cdot L \right ) \geqslant 1,$$
we get:
\begin{align*}
\log \deg(f_n) & 
=  \ch^{-1} \left  ( {f_n^* L} \cdot L \right ) + O(1) \\
& =   d(L, \rho_{f_n}(L)) + O(1) ,
\end{align*}
so we also obtain the central limit theorem for $(\log \deg f_n)$.

Let us now prove it for $(\log \lambda_1(f_n))$. As we just proved, the sequence
\begin{equation} \label{E:final-sum}
\frac{d(L, \rho_{f_n}(L))  - n \ell}{\sqrt{n}}
\end{equation}
converges in distribution to a Gaussian. Note that 
$$\log \lambda_1(f) =  \tau(f)$$
and by Lemma \ref{lem_inequality_gromov}, 
$$\log \lambda_1(f) = d(L, \rho_f(L)) - 2 \langle \rho_f (L), \rho_{f^{-1}} (L) \rangle_L + O(1).$$
Now, by Lemma \ref{L:small-gromov} applied to $\varphi(n) = \sqrt{n}$, 
\begin{equation} \label{E:sqrt-gromov}
\mathbb{P}\left( \frac{ \langle \rho_{f_n} (L), (\rho_{f_n})^{-1}(L) \rangle_L}{\sqrt{n}} \geqslant \epsilon \right) \to 0.
\end{equation}
Hence, by combining \eqref{E:final-sum} and \eqref{E:sqrt-gromov}, we get a CLT for $\log \lambda_1(f_n)$, as required.

Finally, if $\sigma = 0$, then, as in \cite[Proof of Theorem 4.7]{benoist_quint_hyperbolic}, there exists $\lambda$ such that  
\begin{equation} \label{E:tau-arith}
\tau(g) = n \lambda
\end{equation}
for any $g$ in the support of $\mu_n$. This implies that $G$ has arithmetic length spectrum, completing the proof. 
\end{proof}

\bibliographystyle{amsalpha}
\bibliography{references}

\end{document}